\documentclass{article}
\usepackage[utf8]{inputenc}
\usepackage[all,2cell]{xy}

\usepackage[numbers]{natbib}
\usepackage{graphicx}
\usepackage{amsmath}
\usepackage{amsthm}
\usepackage{amsfonts}
\usepackage{amssymb}
\usepackage{indentfirst}
\usepackage{bussproofs}
\usepackage{tikz-cd}
\usepackage{xcolor}
\DeclareMathAlphabet{\mathpzc}{OT1}{pzc}{m}{it}
\usepackage{enumitem} 
\usepackage{geometry}
\usepackage{url}
\usepackage{longtable}
\usepackage{mathabx}

\theoremstyle{plain} 
\newtheorem{theorem}{Theorem}[section]

\newtheorem{proposition}[theorem]{Proposition}

\theoremstyle{definition} 
\newtheorem{definition}[theorem]{Definition}

\newtheorem{remark}[theorem]{Remark}

\newtheorem{corollary}[theorem]{Corollary}

\newcommand{\type}{{\it type}}
\newcommand{\mf}{{\bf MF}}
\newcommand{\hott}{{\bf HoTT}}
\newcommand{\mtt}{\mbox{{\bf mTT}}}
\newcommand{\emtt}{\mbox{{\bf emTT}}}
\newcommand{\mltt}{\mbox{{\bf MLTT}}}

\newcommand{\ourdef}{{\ :\equiv\ }}

\begin{document}

\newgeometry{left= 3.2cm, right= 3.2cm}
\title{On the Compatibility of Constructive Predicative Mathematics with Weyl's Classical Predicativity}

\author{Michele Contente \\ \footnotesize Institute of Philosophy of the Czech Academy of Sciences, Prague, Czechia \\ \and Maria Emilia Maietti \\ \footnotesize Department of Mathematics, University of Padova, Italy}

\date{\today}

\maketitle
\begin{abstract}
\noindent  It is well known that most constructive and predicative foundations  aiming to develop Bishop's constructive analysis are incompatible with a classical predicative development of analysis as put forward by Weyl in his {\it Das Kontinuum}.

 Here, we show how this incompatibility arises from the possibility to define sets by quantifying over (the exponentiation of) functional relations. Such a possibility is present in most constructive foundations but it is not allowed in modern reformulations of Weyl's logical system. In particular, we show how in Aczel's Constructive Set Theory, Martin-L\"of's type theory and Homotopy Type Theory, the incompatibility with classical predicativity \`a la Weyl reduces to the fact of being able to interpret Heyting arithmetic in all finite types  with the addition of the internal rule of number-theoretic unique choice, identifying functional relations over natural numbers with a primitive notion of function defined as  $\lambda$-terms of type theory.
 
\noindent Then,  we argue that a possible way out is offered by constructive foundations, such as  the Minimalist Foundation,
where exponentiation is limited to functions defined as $\lambda$-terms of (dependent) type theory. 
 
\noindent The price to pay is to renounce number-theoretic choice principles, including the rule of unique choice,  typical of most foundations formalizing Bishop's constructive mathematics. This restriction calls for a {\it point-free}  constructive development of topology as advocated by P. Martin-L\"of and G. Sambin with the introduction of Formal Topology.

\noindent We then conclude that the Minimalist Foundation promises to be a natural crossroads between Bishop's constructivism and Weyl's classical predicativity provided that a point-free constructive reformulation of analysis is viable.
%

\end{abstract}

\tableofcontents

\section{Introduction}
The book {\it Das Kontinuum} \cite{Weyl+1918, weyl1994continuum} is a milestone in the foundations of mathematics and, in particular, of classical predicative mathematics. Weyl's monograph is mainly devoted to the predicative reformulation of real analysis.

Another major reformulation of real analysis within a foundation weaker than the usual classical one was given by Bishop in \cite{bishop1967}. While Bishop followed Brouwer in adopting intuitionistic logic, he refrained in his approach from including  Brouwer's continuity principles that are inconsistent with classical mathematics. 
Bishop aimed to develop {\it constructive} mathematics not in opposition to classical mathematics but as a generalization of the latter.

 After the publication of Bishop's book, there was a flourishing of foundations for constructive mathematics aimed at formally capturing Bishop's approach. A relevant example in axiomatic set theory is
Aczel's constructive set theory {\bf CZF}  \cite{aczel78},  and in type theory Martin-L\"of Type Theory {\bf MLTT} \cite{martin-lof:bibliopolis, nordstrom1990programming}  and the more recent Homotopy Type Theory \hott\  in \cite{hottbook} inspired by  Voevodsky's Univalent Foundations. In particular, \hott\  gained a lot of attention
for its application to the formalization of synthetic homotopy theory and higher category theory.
However, these systems are sometimes not even mutually compatible and this consideration led the second author in joint work with G. Sambin to conceive
the Minimalist Foundation (for short \mf)   in \cite{ms05}, later fully developed as a two-level system in \cite{m09}.
The two-level structure of  {\bf MF} in \cite{m09}, which includes an intensional level, an extensional  level, and  an interpretation of the latter into the first,
was introduced to better achieve the compatibility of {\bf MF}  with constructive and classical foundations including those mentioned (see \cite{m09,cm2024}) as well as internal theories of categorical universes such as that of elementary topoi in \cite{LS1986, Maietti05}.

It is also well known that, among all the foundations of constructive and predicative mathematics, at least the mentioned ones, {\bf CZF},  \mltt\ and \hott,
 are incompatible with a classical predicative development of mathematics as advocated by Weyl in  {\it Das Kontinuum}, since
they
 become impredicative when extended with the law of excluded middle.
 
Here we first discuss the source of this incompatibility. The reason lies in their possibility of defining sets by quantifying over (the exponentiation of) functional relations. Such a possibility is not allowed in modern formulations of
 Weyl's system \cite{DBLP:journals/bsl/Avron20, adams2010weyl, DBLP:journals/apal/AdamsL10, feferman1988weyl}.
More formally, we show that  the incompatibility  of  {\bf CZF}, {\bf MLTT} and \hott\
with classical predicativity \`a la Weyl reduces to the fact of being able to interpret Heyting arithmetic in all finite types  with the addition of the internal rule of number-theoretic unique choice, identifying functional relations over natural numbers with a primitive notion of function defined as  $\lambda$-terms of type theory.
 
Then, we argue that it is not necessary to sacrifice exponentiation altogether to reconcile Bishop's mathematics with Weyl's approach if we adopt a foundation like the Minimalist Foundation and this for various reasons.

First of all, the intuitionistic version of the subsystem of second-order arithmetic known as {\bf ACA} \cite{simpson2009subsystems} can be interpreted in the extensional level of \mf\ by preserving the meaning of all the logical operators, as well as the classical {\bf ACA} can be analogously interpreted in the classical extension
of \mf. Therefore, Weyl's reconstruction of analysis can be performed in the classical version of \mf,
since it is known from results of reverse mathematics in \cite{feferman1988weyl, feferman1998light,simpson2009subsystems} that the formalization of Weyl's analysis can be performed at least in {\bf ACA}, whilst the proof-theoretic strength of Weyl's original logical system exceeds it, see \cite{DBLP:journals/bsl/Avron20, adams2010weyl, DBLP:journals/apal/AdamsL10, feferman1988weyl}.

Second, the classical version of \mf\ is equiconsistent, as shown in \cite{talkmillyober, MSdoubleneg}, with the original constructive \mf, whose proof-theoretic strength
is strictly predicative à la Feferman being interpretable in first-order  Martin-L\"of's type theory  with one universe \cite{m09} and directly in Feferman's {\bf $\widehat{ID_1}$} (see \cite{IMMS}).

Third,  in both levels of \mf,  whilst the exponentiation of functional relations does not necessarily form a set,  it does so the exponentiation of a primitive notion of function
defined by $\lambda$-terms. This is possible since, in both levels of \mf, the logic is many-sorted with sorts given by dependent types 
equipped with elements defined by $\lambda$-terms as in Martin-L\"of's type theory.

\noindent The price to be paid in \mf\  is giving up the so-called {\it rule of unique choice} identifying functional relations with $\lambda$-terms. In this way, exponentiation of type-theoretic functions does not entail exponentiation of functional relations as it happens
 instead in Martin-L\"of's type theory and in Homotopy Type Theory, whose logics are also sorted on a dependent type theory. Accordingly, dropping the rule of unique choice entails dropping all the number-theoretic choice principles characteristic of Bishop's mathematics as well as the fact that real numbers à la Dedekind or Cauchy form a set.
 
 As a consequence, these limitations call for a {\it point-free}  constructive development of topology as advocated by Martin-L\"of and Sambin with the introduction of Formal Topology \cite{sambin1987} to avoid the use of Brouwer's Fan theorem as explained in \cite{MSpointfree}. Evidence that such an approach
can be successfully developed is provided by the point-free results in  \cite{CN95,Palmgren05, pal07,spa, kawai2023, sambinbook}.

Given that the point-free topology of Dedekind real numbers and  that of Cantor space are definable in \mf,
if a point-free reformulation of classical analysis were fully viable,
then
  \mf\  would become a natural crossroads between Bishop's constructivism and Weyl's classical predicativity.

In the future, besides determining the exact proof-theoretic strength of  \mf,  which is currently an open problem,
we intend to investigate whether the  extensions of \mf\   with inductive and coinductive definitions in \cite{MMR21, MMR22, MS2023}
are still  equiconsistent with their classical counterparts as established for \mf\ in \cite{talkmillyober, MSdoubleneg}, or at least  these classical counterparts are still predicative.
A positive answer to one of these questions would provide more expressive foundations than \mf\ capable of formalizing a whole development of point-free topology, including for example the results in \cite{mv04,curicind,CPS11,sambinbook}, while still keeping the compatibility with both Bishop's constructive mathematics and Weyl's classical predicative mathematics.

\noindent  The structure of the paper is the following: first, we give a short overview of Weyl's classical predicative conception in {\it Das Kontinuum} and highlight the points of convergence with some ideas underlying the Minimalist Foundation. In Section \ref{sec3} we present some well-known examples of constructive predicative foundations. Then, in Section \ref{sec4} we show that already Heyting Arithmetic in all finite types is incompatible with classical predicativity if it is extended with classical logic and the (number-theoretic) internal rule of unique choice. In Section \ref{sec5} we extend this argument to other systems and analyze the incompatibility of \mltt, \hott\ and {\bf CZF} with classical predicativity. Finally, in Section \ref{sec6} we introduce the Minimalist Foundation and argue that it might represent a natural solution to the issues discussed in the previous sections.

\section{Weyl's Classical Predicativity}\label{sec2}

\noindent {\it Das Kontinuum} \cite{Weyl+1918, weyl1994continuum} can be regarded as the first systematic attempt to develop {\it predicative} mathematics. The book's main focus is the foundation of analysis and the reconstruction of many fundamental results in this area without resorting to impredicative principles. Examples of theorems shown by Weyl in his framework include  {\it sequential} versions of the Heine-Borel Theorem and the Least Upper Bound principle, Uniform continuity for continuous functions over a closed interval, the existence of a maximum and a minimum for them as well as 
the Fundamental Theorem of the Calculus.

Therefore, the main merit of {\it Das Kontinuum} was to show that predicative mathematics was a viable project. Feferman has shown that the above mentioned results can be formalized in the subsystem of second-order arithmetic known as {\bf ACA}$_0$ and further evidence in support of this claim has been provided by the Reverse Mathematics program \cite{simpson2009subsystems}. In this section, we review some basic tenets of Weyl's approach by relating them to the Minimalist Foundation, which will be introduced later in Section \ref{sec6}.

 \noindent Weyl acknowledged - in agreement with Poincar\'e- the primitive character of the natural numbers and the associated notion of iteration. According to Weyl, every attempt to understand natural numbers in terms of more fundamental notions -- as it happens in Zermelo's set theory or in the logicist program of developing arithmetic from purely logical notions-- was doomed to failure because it necessarily presupposes this very notion: 

\begin{quote}
    \lq And I became firmly convinced (in agreement with Poincar\'e, whose philosophical position I share in so few other respects) that the {\it idea of iteration}, i.e., {\it of the sequence of natural numbers, is an ultimate foundation of mathematical thought }- in spite of Dedekind's \lq\lq theory of chains'' which seeks to give a logical foundation for definition and inference by complete induction without employing our intuition of the natural numbers. For if it is true that the basic concepts of set theory can be grasped only through this \lq\lq pure'' intuition, it is unnecessary and deceptive to turn around them and offer a set-theoretic foundation for the concept of \lq\lq natural number''.' ({\it ibid.}, p.48)
\end{quote}

Weyl's conception of set bears some resemblance to the way sets are understood in  Martin-L\"of's type theory ({\bf MLTT})  and also in both levels of the Minimalist Foundation ({\bf MF}), which will be presented in the next sections. In such theories, a set is determined by rules for forming its elements, equipped with an elimination rule that allows proving properties by induction over it and defining functions by recursion from it to another set, see \cite{nordstrom1990programming}.

 Furthermore, Weyl considered the notion of set (and function) underlying the classical foundation of analysis to be irremediably circular.
 Weyl noted that  many basic definitions of analysis (e.g. that of least upper bound) were actually circular, 
 and decided to reformulate analysis by avoiding such viciously circular definitions in favour of a predicative development.
 In particular, such a circularity clearly shows up in Dedekind's construction of the real numbers: real numbers defined in this way do not form an {\it extensionally determinate collection}\footnote{See \cite{DBLP:journals/bsl/Avron20, crosilla2023weyl} for a detailed analysis of this concept.}. Weyl elaborates on this point in a letter to H\"older, published as appendix in \cite{weyl1994continuum}: 

\begin{quote}
    \lq If $A$ is a property of properties, then one may form the property $P_A$ which belongs to an object $x$ if and only if there is a property constructed by means of these principles which belongs to $x$ and itself possesses the property $A$. That would be a blatant {\it circulus vitiosus}; yet our current version of analysis commits this error and I consider it ground for censure.' (\cite{weyl1994continuum}, p.113)
\end{quote}

In particular, it follows from this quotation that the power of a set is not a set according to Weyl. Thus we cannot define a new set by quantifying over the powerset to which it should belong. The subtle point is that this does not mean that the power of a set does not exist, but simply that it is not a set but a collection that is not inductively generated.

In the Minimalist Foundation, this distinction between sets and collections is explicitly assumed. Accordingly, in Weyl's work collections exceed the power-collection of a set, as shown in \cite{adams2010weyl, DBLP:journals/apal/AdamsL10, DBLP:journals/bsl/Avron20}, but they are not allowed to enter into the definition of new sets.

A major distinction between sets and collections in the Minimalist Foundation is that, whilst both entities are generated by
rules forming their canonical elements as sets in Martin-L\"of's type theory, sets distinguish from collections for being equipped with an
induction principle and  a recursive definition of functions from them toward any other entity. Collections instead do not enjoy any induction principle
or recursive definition from them, since inductive or recursive definitions acting on them would not be {\it invariant}
under extensions of the theory with new sets. This approach aligns closely with the idea that predicative sets should be invariant, which traces back to Poincar\'e and has been carefully analyzed in \cite{Crosilla2022}.

Both in {\it Das Kontinuum} and in the Minimalist Foundation real numbers provide an example of a collection which is not a set. In particular, Weyl introduced real numbers  as Dedekind cuts of rationals that are formalized as quadruples of natural numbers. The fact that the resulting collection $\mathbb{R}$ of real numbers is not {\it extensionally determinate} in Weyl's terminology, means that {\it we cannot define new sets by quantification over} $\mathbb{R}$, in sharp contrast with the case of natural numbers. This implies, for instance, that the least upper bound principle for sets of real numbers is not predicatively provable, while the principle is provable for sequences of real numbers, since the latter do not involve quantification over $\mathbb{R}$.

\noindent Finally, we want to point out that Weyl in {\it Das Kontinuum} places a significant emphasis on his notion of function and contrasts it with the traditional set-theoretical definition (see \cite{weyl1994continuum}, pp.33-34). According to Avron's analysis in \cite{DBLP:journals/bsl/Avron20}, for Weyl, a function is always determined by a rule. In other words, a function cannot be given by an arbitrary relation satisfying the functionality condition, as in set theory; rather, there are syntactic constraints on the definability of the relation that must be respected. 

\noindent  In this regard,
it might be intriguing to relate Weyl's conception of function as a rule-based object to a computational understanding of functions. 

Weyl's insights about functions as rule-based objects align, in some way, with the fact that the Minimalist Foundation distinguishes the notion of function as functional relation from that of function as a primitive operation defined by a $\lambda$-term, resembling the Brouwer's notion of  lawlike sequence (see \cite{sambin2008continuity}).
The distinct character of these two notions of function is a peculiarity of {\bf MF} among foundations for constructive predicative mathematics.  While it might be argued that Weyl foreshadowed this distinction, nevertheless his technical treatment of the notion of function is not exactly the same. 

 There have been various attempts in the literature to formalize Weyl's system in a modern fashion, since Weyl's original presentation presents some ambiguities. An influential reconstruction of Weyl's system has been carried out by Feferman in several works \cite{feferman1988weyl, feferman1998light, feferman2000significance}. Feferman has extensively argued that Weyl's system should correspond to the theory $\mathbf{ACA}_0$ or, if the full induction scheme is accepted, to $\mathbf{ACA}$, which are well-known subsystems of second-order arithmetic \cite{simpson2009subsystems}. But such a reconstruction 
has been criticized by Adams and Luo in \cite{adams2010weyl, DBLP:journals/apal/AdamsL10}, and more recently by Avron in \cite{DBLP:journals/bsl/Avron20}. They observed that in certain sections of {\it Das Kontinuum}, Weyl introduces non-arithmetical sets and makes use of induction for non-arithmetical formulas. Therefore, they concluded that Weyl's system should be stronger than Feferman's reformulation of Weyl's system, even if ${\bf ACA}_0$ should be sufficient to reconstruct predicative analysis as developed in {\it Das Kontinuum}. 

The systems introduced in the above-mentioned papers 
are many-sorted logics axiomatized over a dependent type theory, in a way similar to both levels of the Minimalist Foundation. They all differ in the way 
 exponentiation of functions is treated.
Indeed,
 Adams and Luo's system in \cite{adams2010weyl, DBLP:journals/apal/AdamsL10}  includes some exponentiation of $\lambda$-functions in the form of function types, which are treated as {\it large types} in the sense that they are not contained in the universe of small types, that in turn can be thought of as the universe of sets. Instead, Avron in \cite{DBLP:journals/bsl/Avron20} restricts the possibility of forming exponentials by assuming that only certain types might serve as codomains by not allowing to define sets by quantification over functions.  On the contrary, in {\bf MF} there is no exponentiation of functions between collections
 but only sets exponentiating $\lambda$-functions between sets.

\section{Constructive Predicative Theories}\label{sec3}

\noindent Among the most popular frameworks introduced as foundations for Bishop's constructive mathematics there are  Constructive Zermelo-Fraenkel set theory $\mathbf{CZF}$ \cite{aczel78} and  Martin-L\"of Type Theory \mltt\  \cite{nordstrom1990programming}. More recently, an extension of the latter called Homotopy Type Theory \hott\ \cite{hottbook} has gained much attention, especially for its applications to synthetic homotopy theory and higher categories. All these systems are examples of constructive and predicative theories. However, the predicativity of these theories is closely tied to the adoption of a constructive logic. Indeed, whenever they are extended with the law of excluded middle or equivalent classical principles, they become impredicative and thus they are not compatible with classical predicativity à la Weyl. Briefly, the issue is that such theories are closed under the {\it exponentiation of functional relations} and the two-element set that, when classical logic is added, are strong enough to allow for impredicative constructions. Therefore, we cannot establish the predicativity of their classical counterparts, as it happens instead in the case of Heyting arithmetic  or the Minimalist Foundation, which we will introduce in detail in Section \ref{sec6}. 

\subsection{Aczel's Constructive Zermelo-Fraenkel set theory}
\noindent {\bf CZF} \cite{aczel78} is a theory formulated in the same language of classical {\bf ZF} and based on intuitionistic logic. The theory includes the following axioms: {\it extensionality}, {\it pairing}, {\it union}, {\it restricted separation}, {\it infinity}, {\it  $\in$-induction}, {\it strong collection}, and {\it subset collection}.

The predicative nature of the theory is reflected in the restriction of {\it separation to $\Delta_0$-formulas}, which means that the defining formula $\phi$ might quantify only over sets that have already been introduced, and in the rejection of the full powerset axiom. The powerset axiom is replaced by the weaker axiom of subset collection, which in turn implies the exponentiation axiom for functions. Functions are defined as functional relations in the standard set-theoretic manner.

We recall that the Axiom of Choice in the form

\begin{equation*}
 ({\bf AC})\    \forall_{(x \in A)}\ \exists_{(y \in B)}\ R(x,y)\ \rightarrow\ \exists_{(f \in A \rightarrow B)}\ \forall_{(x \in A)}\ R(x,f(x))
\end{equation*}
where $A$ and $B$ are sets and $R$ any relation on them, 
cannot be added to {\bf CZF} while preserving its constructive character: 

\begin{theorem}
    {\bf CZF} + {\bf AC} derives the law of excluded middle $\phi \vee \neg \phi $ for $\Delta_0$-formulas $\phi$.
\end{theorem}
\begin{proof}
    See \cite{AczelRath}.
\end{proof}
\begin{remark}
    In \cite{AczelRath}, it is shown that this proof can be carried out in a weaker subsystem of {\bf CZF}. On the other hand, the weaker principles of {\it countable choice} and {\it dependent choice} can be added to {\bf CZF} without leading to undesirable results.
\end{remark}
Moreover, in {\bf CZF} real numbers form a set:
\begin{theorem}
Real numbers both as Dedekind cuts or as Cauchy's sequences form a set in {\bf CZF}.
\end{theorem}
\begin{proof}
    See \cite{AczelRath}.
\end{proof}

\subsection{Martin-L{\"o}f's type theory}
Martin-L{\"o}f Type Theory (\mltt) \cite{nordstrom1990programming} is a foundation for constructive mathematics, which was explicitly introduced to formalize Bishop's mathematics \cite{ml1975}. The theory is not presented in the usual language of first-order logic, but it is formulated in the language of dependent type theory.

Types of \mltt\ are introduced through the following kinds of rules: formation, introduction, elimination, and computation. In particular, the constructors of a given type are specified by the introduction rules, while the elimination rule states simultaneously a principle of induction at the level of types  and of definition by recursion at the level of terms with the aid of the computation rules. The predicative nature of the theory is related to the fact that all the types, also called {\it sets}, are inductively generated by the relevant rules. 
Two fundamental aspects of \mltt\ are the presence of a {\it primitive notion of function}, called here {\it $\lambda$-function} to disambiguate it  from the notion of function as {\it functional relation}, and the {\it identification between propositions and types}. $\lambda$-functions are represented as $\lambda$-terms and this differentiates this foundational theory from other theories such as {\bf CZF}. The identification between propositions and types makes the constructive logic of \mltt\ stronger than usual intuitionistic logic. This additional strength is visible in the case of the existential quantifier that is identified with the dependent sum type $\Sigma$. The elimination rule for this type is strong enough to allow for the definition of projection operators\footnote{We recall that the canonical elements of $\Sigma$-types given by the introduction rule are pairs.}. These projections play a fundamental role in establishing the following fact: 

\begin{theorem}
    The type-theoretic Axiom of Choice 
\begin{equation*}
    ({\bf AC})\ \Pi_{(x:A)}\ \Sigma_{(y:B)}\ R(x,y)\ \rightarrow \Sigma_{(f: A \rightarrow B)}\ \Pi_{(x:A)}\ R(x,f(x))
\end{equation*}
where $A$ and $B$ are  sets (i.e. types) and $R$ any relation on them, 
    is derivable in \mltt.
\end{theorem}
We remark that the rules for $\Sigma$-type are strong enough to derive also the corresponding {\it choice rule}. Furthermore, the Axiom of Unique Choice {\bf AC!}, which identifies {\it functional relations} with  {\it $\lambda$-functions},
\begin{equation*}\label{unique}
 ({\bf AC!})\    \Pi_{(x:A)}\ \Sigma !_{(y:B)}\ R(x,y)\ \rightarrow\ \Sigma_{(f: A \rightarrow B)}\ \Pi_{(x:A)}\ R(x,f(x))
\end{equation*}
\vspace{0.3em}
(where $\Sigma !_{(y:B)}\ R(x,y)\ \ourdef\ \Sigma_{(y: B)}\ R(x,y)\ \times\ \Pi_{( y_1, y_2 :B)}\ (R(x,y_1)\ \times\ R(x,y_2))\ \rightarrow\ \mathsf{Id}(B, y_1, y_2)$) 
also follows from the validity of {\bf AC}.

Henceforth, we can refer in \mltt\  with the name {\it function}  to either a {\it functional relation} or a {\it $\lambda$-function}  without ambiguity. As anticipated in section~\ref{sec2}, the two notions of function are instead kept distinct in the Minimalist Foundation as well as in some modern presentations of Weyl's system as those given in \cite{DBLP:journals/apal/AdamsL10, DBLP:journals/bsl/Avron20}.

Due to the absence of quotient sets in {\bf MLTT}, real numbers, both à la Dedekind or à la Cauchy, are formally represented there using {\it setoids} whose elements are then indexed by a set. A setoid \cite{palmgren2005bishop}, called {\it set}
by Bishop in (\cite{bishop1967}, p. 13), is a pair  $(A, =_A)$,  consisting of a set 
called {\it support} equipped with a suitable equivalence relation $a=_A a' \ [a,a': A]$ . 

For example, Bishop reals in \cite{bishop1967} can be
represented as the setoid whose support is given by  the set of Cauchy regular sequences 
of rational numbers equipped with the equivalence relation $=_{\mathcal{R_{C}}}$ as in (\cite{bishop1967}, pp.15-16).
Bishop defined them by using the notion of {\it operation} which we can identify in \mltt\  with  the notion of $\lambda$-function, i.e.  with a $\lambda$-term of {\bf MLTT}.
Moreover, for what we have said about the notion of function in \mltt,
 {\it Cauchy reals} defined as Cauchy functional relations can be identified  in \mltt\ with Bishop reals. For unifying the various concepts, we name them {\it Cauchy sequences à la Bishop}.


\noindent Instead, in \mltt\ Dedekind reals can be defined only relatively to a universe of sets  $U_k$ (indexed by a natural number $k$) and are represented by a setoid whose support is given by the set of Dedekind cuts on the rationals relative to $U_k$ as follows: 
\begin{definition}\label{dedcut}
    A Dedekind cut on the rationals is a pair $(L,U)$ of subsets $L,U \subseteq \mathbb{Q}$ that satisfy the usual conditions of being inhabited, disjoint, open, monotone, and located, all defined as in \cite{maiettisambinhand}.
\end{definition}

Notice that the subsets $L,U$ are defined as elements of the setoid of  propositional functions $L,U: \mathbb{Q} \rightarrow \mathcal{U}_k$ for some universe $\mathcal{U}_k$  quotiented under equiprovability.  Therefore, \lq $a\ \varepsilon\ L$' means that there exists a proof-term of $L(a)$.

\begin{definition}
We denote with $\mathbb{R}_{D_{k}}$  the support of Dedekind cuts made
of propositional functions with values in the universe $U_k$ for any natural number $k$.
\end{definition}

Again thanks to the validity in \mltt\ of  choice principles on natural numbers and  thanks to the exponentiation of $\lambda$-terms by means of the dependent product type $\Pi_{(x:A)}B(x)$ of any family of types
$B(x) (x:A)$, in
{\bf MLTT} each real number à la Dedekind in the form of a Dedekind cut corresponds bijectively
to a Bishop real and hence to a Dedekind cut in $\mathbb{R}_{D_{0}}$. 

\begin{proposition}\label{bijdedcau}
    In \mltt, Dedekind reals $\mathbb{R}_{D_{0}}$ are in bijective correspondence with the Cauchy reals, defined as Cauchy sequences à la Bishop.
\end{proposition}
\begin{proof}
See  Prop.17 in \cite{CN95} and section 4.3 of \cite{MSpointfree}. 
\end{proof}

As a consequence, we deduce that:

\begin{proposition}\label{dedres}
 In \mltt,  the set $\mathbb{R}_{D_{0}}$ of Dedekind reals is   in bijective correspondence with the set of Dedekind reals $\mathbb{R}_{D_{k}}$  for any natural number $k$.
\end{proposition}

\begin{proof}
   Given a Dedekind cut as in def.~\ref{dedcut} in $\mathbb{R}_{D_{k}}$  for any natural 
   number $k$, by using the axiom of choice we can define a Bishop real given by an operation
   $\lambda n. x_n $, usually written $(x_n)_{n}$, as in Prop. 17 in \cite{CN95} or section 4.3 in \cite{MSpointfree}, whose 
   cut $(L_0, D_0)$ in $\mathbb{R}_{D_{0}}$ is defined as

$$L_0\equiv \{ p : \mathbb{Q}\ \mid\ \Sigma_{n : \mathbb{N}}\ p\leq x_n-2/n \ \}\qquad\qquad  U_0\equiv \{ q: \mathbb{Q}\ \mid\ \Sigma_{n: \mathbb{N}}\ q\geq x_n +2/n \ \}$$
  
  \noindent and is propositionally equivalent to the starting cut $(L,U)$ in $\mathbb{R}_{D_{k}}$ thanks
   to the correspondence between Dedekind cuts, formal points of point-free topology of real numbers and Bishop reals in 
   section 4.3 of \cite{MSpointfree}.
\end{proof}

\begin{remark}
  The previous proposition resembles a special case of  Russell's Reducibility Axiom which constructively requires  a  choice principle to be derived.  If the law of excluded middle were valid the full Reducibility axiom would follow as shown in section 5. 
\end{remark}

\subsection{Homotopy Type Theory}
\noindent Homotopy Type Theory (\hott) \cite{hottbook} is an extension of \mltt\ and is an instance of Voevodsky's Univalent Foundations. It extends \mltt\ with the {\it Univalence Axiom} and {\it Higher Inductive types}, examples of which include quotients of homotopy sets and propositional truncation. The Univalence Axiom is formulated  as follows: 
\begin{equation*}
    \mbox{the map}\ \mathsf{idtoeqv}: (A =_{\mathcal{U}} B)\ \rightarrow\ (A \simeq B)\ \mbox{is an equivalence}
\end{equation*}
where $(A \simeq B)$ is the type of equivalences between $A$ and $B$. The map $\mathsf{idtoeqv}$ and the notion of equivalence $\simeq$ are defined as in Section 2 of \cite{hottbook}. Among the consequences of this axiom, there are certain extensional principles (e.g. function extensionality) that are not available in intensional \mltt. 

A significant feature of \hott\ is allowing for a representation of logical notions alternative to the propositions-as-types paradigm. The key notion here is that of {\it h-proposition}, where a type $P$ is a h-proposition whenever for all $x,y:P$, $x =_P y$ holds, and it is analogous to the notion of {\it mono type} in the internal theory of topos in \cite{Maietti05}. Further, a type $A$ is a {\it h-set} if the associated identity type $=_A$ is a h-proposition. 

The logic of h-propositions behaves differently from the propositions-as-types logic and is similar to the internal logic of elementary topoi or of regular categories as presented in \cite{Maietti05}. In particular, the existential quantifier is not identified with the $\Sigma$-type any longer. Rather, it is identified with the propositional truncation of the $\Sigma$:
\begin{equation*}
    \exists_{(x:A)}\ P(x)\ \ourdef\ \vert\vert\ \Sigma_{(x:A)}\ P(x)\ \vert\vert
\end{equation*}
Observe that the propositional truncation of a type is always a h-proposition. Indeed the constructors for the propositional truncation of a type $A$ are the following: if $a:A$, then $\vert\ a\ \vert: \vert\vert\ A\ \vert\vert$ and  $x =_{\vert\vert A\vert\vert} y$ holds for all $x,y: \vert\vert\ A\ \vert\vert$.
Moreover, it holds that if $P$ is a h-proposition, then it is equivalent to its truncation.

\noindent Observe also that  in \hott,  while it is obviously possible to derive the type-theoretic {\bf AC} under the identification 
propositions-as-types,
it is not instead possible to derive the type-theoretic {\bf AC} 
using the identification of propositions as h-propositions in the version:
\begin{equation*}
    \Pi_{(x:A)}\ \exists_{(y:B)}\ R(x,y)\ \rightarrow\ \exists_{(f: A \rightarrow B)}\ \Pi_{(x:A)}\ R(x,f(x))
\end{equation*}
with $A, B$ h-sets and $R(x,y)$ h-proposition for $x:A, y:B$, as shown in \cite{hottbook}.
Indeed, in \hott\ it is only possible to derive the axiom of unique choice {\bf AC!} formulated as follows 
\begin{equation*}
 ({\bf AC!})\    \Pi_{(x:A)}\ \exists !_{(y:B)}\ R(x,y)\ \rightarrow\ \Sigma_{(f: A \rightarrow B)}\ \Pi_{(x:A)}\ R(x,f(x))
\end{equation*}

\noindent
with $A, B$ types and $R(x,y)$ h-proposition and where unique existence $\exists !_{(y:B)} R(x,y)$ amounts to the following condition: 
\begin{equation*}
    \exists_{(y:B)}\ R(x,y)\ \times\ \Pi_{(y_1, y_2: A)}\ R(x,y_1)\ \land\ R(x,y_2)\ \rightarrow\ y_1 =_B y_2.
\end{equation*}

\begin{proposition}\label{unipri}
    {\bf AC!} can be derived in \hott.
\end{proposition}

 \begin{proof}
    See \cite{rijke2015sets}.
\end{proof}

\begin{remark}
    Unique existence can be formulated in an equivalent way in \hott\ by employing the notion of contractible type. A type $A$ is said to be contractible if the following type is inhabited: 

    \begin{equation*}
        \Sigma_{(x:A)}\ \Pi_{(y:A)}\ x =_A y.
    \end{equation*}
Therefore, we can rewrite unique choice by requiring the contractibility of $\Sigma_{(y:B)}\  R(x,y)$. 
    
\end{remark}

\begin{remark}
    The proof of {\bf AC!} in \cite{rijke2015sets}  is a special case of a result that holds in the internal dependent type theory $\mathcal{T}_{reg}$ of any regular category \cite{Maietti05}. Following \cite{Maietti05}, observe that the propositional truncation of a $\Sigma$-type $\Sigma_{(x:A)}\ P(x)$ can be equivalently expressed in $\mathcal{T}_{reg}$  as a quotient, namely the quotient $\Sigma_{(x:A)}\ P(x)/\mathbf{1}$ over the unit type identifying all the proof-terms. 
    Then, the proof in \cite{rijke2015sets} corresponds to the argument in \cite{Maietti05} relying on the fact that unique existence can be written as $\Sigma_{(y:B)}\ R(x,y)$ together with the usual uniqueness condition. In this case, $\Sigma_{(y:B)}\ R(x,y)$ is a mono type, according to the terminology introduced in \cite{Maietti05}, and is equivalent to its quotient over the unit type, and hence we can use $\Sigma$-projections to prove {\bf AC!}.
\end{remark}

\noindent A crucial feature of \hott\ is the presence of quotient sets as higher inductive types. This means that Cauchy real numbers à la Bishop can be formalized as a quotient of the type of Cauchy sequences under the equivalence relation $=_{\mathbb{R}_C}$ without resorting to setoids. In particular, they form a set in \hott.

\noindent The case of Dedekind reals is more delicate. The notion of cut can be formalized as a pair of propositional functions $(L,U): \mathbb{Q} \rightarrow \mathsf{Prop}_{\mathcal{U}_k}$ relative to some univalent universe of set $U_k$, that satisfy the conditions of Def.\ref{dedcut}. So it is possible to show that the Dedekind cuts up to a given universe form a set, but we cannot collect into a set all the Dedekind cuts varying in any universe unless impredicativity in the form of {\it propositional resizing} is assumed. Indeed, for the absence of dependent choice  Prop.\ref{dedres} fails in \hott, and hence there is no reduction of Dedekind reals of level $k$ to those of the lowest level. 

\noindent This fact can be contrasted with what happens in the Minimalist Foundation, and in particular in its extensional level where the presence of the power {\it collection} of a set 
allows us to collect all Dedekind cuts together inside the theory.

\noindent In any case, it must be noticed that real numbers à la Cauchy  in \hott\ are defined
in a different way from the usual type-theoretic definition of Cauchy reals or Bishop reals.
Indeed,  they are defined as a  higher inductive type $\mathbb{R}_{\mathsf{c}}$   (\cite{hottbook}, see Sec. 11.3)  capturing  the idea that Cauchy reals can be regarded as {\it the free complete metric space generated by the rationals} and they form an h-set.

\section{The incompatibility of ${\bf HA}^\omega + {\bf iRC!}_{\mathbf{N},\mathbf{N}}$ with Weyl's predicativity}\label{sec4}

\noindent In this section, we want to show that already the weaker system of  Heyting Arithmetic in all finite types (${\bf HA^\omega}$), when extended with the number-theoretic internal rule of unique choice and the law of excluded middle, becomes impredicative. Since the most common constructive and predicative foundations interpret this system, then the result can be easily extended to these more expressive theories. 

\noindent For the presentation of  ${\bf HA^\omega}$, we refer to the standard reference \cite{TVD88}.

The finite type structure $\mathrm{T}$ is defined as follows:  $\mathbf{N}$ is a finite type and if $\sigma, \tau$ are finite types, then so are $\sigma \rightarrow \tau$ and $\sigma \times \tau$.

The language of ${\bf HA^\omega}$ includes countably infinite variables $x^\sigma, y^\sigma, z^\sigma, \ldots$ for each type symbol $\sigma \in \mathrm{T}$. A binary predicate for equality $=_\sigma$ at each type $\sigma \in \mathrm{T}$ and an application operator $\mathsf{Ap}_{\sigma, \tau}$ for all $\sigma, \tau \in \mathrm{T}$. Furthermore, the language comprehends constants for zero and successor, recursor, pairing and projections and the combinators ${\bf k}$ and ${\bf s}$. Note that the logical implication is denoted with the symbol $\supset$ to avoid confusion with the symbol $\rightarrow$ used to denote function types.

The axioms of ${\bf HA^\omega}$ include the axioms of many-sorted intuitionistic predicate logic with equality, congruence rules for equality, the usual axioms for zero and successor with the induction scheme, and the defining conditions for the constants.

We recall two fundamental results about ${\bf HA^\omega}$. It is immediate to realize that Heyting Arithmetic ({\bf HA}) is a subsystem of ${\bf HA^\omega}$. Furthermore, if we extend ${\bf HA^\omega}$ with the Axiom of Choice for all finite types $\sigma, \tau$:

\begin{equation*}
    ({\bf AC}) \ \ \forall x^\sigma \ \exists y^\tau\  \phi(x,y)\ \supset\ \exists f^{\sigma \rightarrow \tau}\  \forall x^\sigma\  \phi(x, fx)
\end{equation*}

\noindent the resulting theory is conservative over {\bf HA}. This result is known as Goodman's Theorem \cite{goodman1978}. 

\noindent Another significant fact is that ${\bf HA^\omega}$ is equiconsistent with its classical version by a double negation translation. This result can be found in \cite{troelstra1973}. However, this result does not extend to the case of ${\bf HA^\omega} + {\bf AC}$. Indeed, it is possible to show that ${\bf HA^\omega} + {\bf AC} + {\bf LEM}$, where {\bf LEM} is the law of excluded middle $\phi \vee \neg \phi$ for all formulas,  is strong enough to derive the second-order comprehension axiom and thus to interpret full second-order classical arithmetic ${\bf PA^2}$, which is an impredicative theory.   The first to observe that ${\bf HA^\omega}$ extended with the Axiom of Choice and classical logic allows one to derive the second-order comprehension principle was Spector in \cite{spector1962provably}. We remark that already the {\it internal number-theoretic rule of unique choice}  is sufficient to carry out the proof. 

\begin{definition}
The {\em internal rule of unique choice}, called ${\bf iRC!}$,  is formulated for ${\bf HA^\omega}$ as follows: 
\begin{equation*}
    \mbox{if}\ \ {\bf HA^\omega} \vdash \forall x^\sigma \ \exists ! y^\tau \ \phi(x,y) \qquad  \mbox{then}\ \ {\bf HA^\omega} \vdash \exists f^{\sigma \rightarrow \tau}\ \forall x^\sigma \ \phi(x,fx)
\end{equation*}
where $\exists ! y^\tau \ \phi(x,y)\ \ourdef\ \exists y^\tau \ \phi(x,y) \land \forall y_1^\tau \ \forall y_2^\tau \ (\ \phi(x,y_1) \ \land\  \phi(x,y_2)\ \supset\ y_1 =_\tau y_2)$. 
\vspace{0.6em}

The special case when ${\bf iRC!}$ acts only on natural numbers, namely both $\sigma$ and $\tau$ coincide with the type $\mathbf{N}$ of natural numbers, 
is called {\it internal number-theoretic rule of unique choice} and it is denoted with the symbol ${\bf iRC!_{\mathbf{N},\mathbf{N}}}$.
\end{definition}

\begin{remark}
    We call the rule displayed above {\it internal} rule of unique choice to contrast it with the following {\it rule of  unique choice}: if $ {\bf HA^\omega} \ \vdash\  \forall x^\sigma \ \exists ! y^\tau\
     \phi(x,y)$, then there exists a function $f^{\sigma \ \rightarrow \ \tau}$ such that $ {\bf HA^\omega} \ \vdash\ \forall x^\sigma \ \phi(x, fx)$,  where the existence of the choice function is stated externally.
    Admittedly, in  $ {\bf HA^\omega}$ the two rules are equivalent by the results in \cite{troelstra1973}. In general, this equivalence holds for all the systems that satisfy the existence property, but it cannot be stated in general. For instance, it does not hold for classical systems. 

\end{remark}

\noindent The theory ${\bf PA}^2$ is usually formulated as an extension of the language of ${\bf PA}$ by a second sort of variables $X, Y, Z \ldots$ standing for subsets of natural numbers and a membership symbol $\in$, where $t \in X$ is a new atomic formula for $t$ first-order term. The axioms include the second-order induction axiom and the full scheme of comprehension for sets of natural numbers. However, the theory can be equally formulated in a language with  typed primitive function symbols. Then first-order variables are interpreted as variables of type $\mathbf{N}$, second-order variables are interpreted by their characteristic functions with type $\mathbf{N} \rightarrow \mathbf{N}$ and the atomic formulas $t \in X$ are interpreted by $X(t) = 1$. In particular, the comprehension axiom is formulated as follows: 

\begin{equation*}
    ({\bf CA})\ \ \exists f^{\mathbf{N}\ \rightarrow\ \mathbf{N}} \ \forall x^\mathbf{N}\ \ (\ fx = 1\ \Leftrightarrow\ \phi(x)\ )
\end{equation*}
where $\phi(x)$ is an arbitrary formula of the language provided that $f$ does not occur free in $\phi$.

Now, we show that in ${\bf HA^\omega} + {\bf iRC!_{\mathbf{N},\mathbf{N}}} + {\bf LEM}$ this comprehension principle is derivable. 

\begin{proposition}\label{CA} The theory ${\bf HA^\omega} + {\bf iRC!_{\mathbf{N},\mathbf{N}}} + {\bf LEM}$ derives the second-order comprehension principle ${\bf CA}$.
\end{proposition}

\begin{proof}
    Let us consider a formula $\phi(x)$ with $x$ of type $\mathbf{N}$. Then, we can define the following relation 

    \begin{equation*}
        \chi_\phi (x,y)\ \ourdef\ (\phi(x)\ \land\ y^\mathbf{N} = 1)\ \lor\ (\neg \phi(x)\ \land\ y^\mathbf{N} = 0)
    \end{equation*}

    By the validity of {\bf LEM}, we can argue by cases: 

    \begin{itemize}
        \item[--] if $\phi(x)$ holds, then $\chi_\phi (x,1)$ holds.

        \item[--] if $\neg \phi(x)$ holds, then $\chi_\phi(x,0)$ holds. 
    \end{itemize}

    We recall that $0 \neq 1$ follows from the axioms of ${\bf HA^\omega}$. Thus we obtain $\forall x^\mathbf{N}\ \exists ! y^\mathbf{N}\ \chi_\phi (x,y)$, that is a functional relation. 

    Therefore, we can apply ${\bf iRC!}_{\mathbf{N}, \mathbf{N}}$  and we obtain that 

    \begin{equation*}
        \exists f^{\mathbf{N} \rightarrow \mathbf{N}}\ \forall x^\mathbf{N}\ \chi_\phi (x, fx)
    \end{equation*}

    which, unfolding the definition of $\chi_\phi$, amounts to 

    \begin{equation*}
        \exists  f^{\mathbf{N} \rightarrow \mathbf{N}}\ \forall x^\mathbf{N}\ ((\phi(x)\ \land\ fx=1)\ \lor\ (\neg \phi(x)\ \land\ fx = 0))
    \end{equation*}

    From which {\bf CA} follows. 
\end{proof}

\begin{corollary}\label{inchao}
    The theory ${\bf PA}^2$ can be interpreted within ${\bf HA^\omega} + {\bf iRC!_{\mathbf{N},\mathbf{N}}} + {\bf LEM}$, which is then impredicative and incompatible with Weyl's predicativity.
\end{corollary}

\section{The incompatibility of  \mltt, \hott\ and {\bf CZF} with Weyl's predicativity}\label{sec5}

\noindent In this section, we show that the classical extensions of the systems presented in Section \ref{sec3} are not compatible with classical predicativity. 

\noindent By {\it compatibility of a theory $\mathsf{T}_1$ with another theory  $\mathsf{T}_2$} we mean that {\it there exists a translation from $\mathsf{T}_1$ to $\mathsf{T}_2$ preserving the meaning of logical and set-theoretical constructors}.   In this manner, a proof of a theorem in $\mathsf{T}_1$ can be transported into $\mathsf{T}_2$ with the same meaning. 

It is easy to see that ${\bf HA^\omega}$ extended with ${\bf iRC!_{\mathbf{N},\mathbf{N}}}$  is compatible both with \mltt\ and \hott\  since these are essentially extensions of ${\bf HA^\omega}$
validating ${\bf iRC!_{\mathbf{N},\mathbf{N}}}$.  

\begin{definition} A compatible translation from ${\bf HA}^\omega + {\bf iRC!_{\mathbf{N},\mathbf{N}}}$ to \mltt\ can be defined as follows: 

\begin{itemize}
    \item[-] $(\mathbf{N})^\ast\ \ourdef\ \mathbb{N}$, while $(\sigma \rightarrow \tau)^\ast\ \ourdef\ \sigma^\ast \rightarrow \tau^\ast$ and $(\sigma \times \tau)^\ast\ \ourdef\ \sigma^\ast \times \tau^\ast$. 

    \item[-] variables $x: \sigma$ are translated as variables $x: \sigma^\ast$ and application terms as application terms in \mltt.

    \item[-] constants of ${\bf HA}^\omega$ are translated as the corresponding terms in \mltt: e.g., $(\textbf{k})^\ast \ourdef \lambda x. \lambda y.x: \sigma^\ast \rightarrow (\tau^\ast \rightarrow \sigma^\ast)$.

    \item[-] the formulas of many-sorted intuitionistic logic are sent to the corresponding types according to the propositions-as-types principle: e.g., $(\exists x^\sigma\ \phi(x))^\ast\ \ourdef\ \Sigma_{(x:\sigma^\ast)}\ \phi^\ast(x)$. 
    
\end{itemize}
\end{definition}

\begin{proposition}\label{hamltt}
    The translation $(-)^\ast$ is a compatible translation of ${\bf HA}^\omega + {\bf iRC!_{\mathbf{N},\mathbf{N}}}$ in \mltt.
\end{proposition}

\begin{proof}
    The proof is by induction over the structure of derivations. In particular, the interpretation of the number-theoretic internal rule of unique choice follows from the validity of the more general rule of choice in \mltt. 
\end{proof}

Therefore,  suppose to extend \mltt\ with the law of excluded middle {\bf LEM} formulated as the existence of a proof-term of the type
$$A+ (A\rightarrow N_0)$$
where  $A$ is any type, $N_o$ is the empty set interpreting the falsum constant, $+$ is the symbol for the disjoint sum type and $\rightarrow$  is the symbol for the function type.  Then in \mltt+{\bf LEM}
we can reproduce the argument showing the
incompatibility of {\bf HA$^\omega$} with Weyl's predicativity:

\begin{corollary}
    \mltt\ + {\bf LEM} becomes impredicative.
\end{corollary}
\begin{proof}
It follows from Prop.~\ref{hamltt} and Cor.~\ref{inchao}.
\end{proof}

\begin{remark}
    In general, in \mltt\ with classical logic we can use the type of booleans {\bf 2} as a classifier, however we do not get a (strong) power object. For instance, the type $\mathbb{N} \rightarrow {\bf 2}$ behaves more like a {\it weak power object}, because \mltt\ lacks of the relevant extensionality principles. 
\end{remark}

\noindent The translation of ${\bf HA}^\omega$ into \mltt\ defined above can be refined in the case of \hott. A finite type $\sigma$ is translated as a pair $\sigma^\ast\ \equiv\ (\sigma^\star, p_{\sigma^\star})$, where the second component is a proof that $\sigma^\star$ is a h-set in \hott. Similarly, formulas of ${\bf HA}^\omega$ can be translated as h-propositions in \hott, while the rest of the translation is defined as in the case of \mltt. Therefore, the following fact holds: 

\begin{proposition}\label{comphot}
     The translation $(-)^\ast$ is a compatible translation of ${\bf HA}^\omega + {\bf iRC!_{\mathbf{N},\mathbf{N}}}$ in \hott.
\end{proposition}

As in the case of \mltt, we can reproduce the same proof of Prop.\ref{CA} but with a significant difference. The law of excluded middle cannot be formulated according to the propositions-as-types identifications for {\it all} types in \hott, because this would yield a contradiction in presence of Univalence (see Ch.3 \cite{hottbook}). However, it is still possible to formulate ${\bf LEM}$  for all those types that are h-propositions as
$$\phi \vee \neg \phi$$
for all h-propositions $\phi$.

This formulation serves our purposes and is consistent with Univalence. 

\begin{corollary}
    \hott\ with the addition of the law of excluded middle for h-propositions becomes impredicative.
\end{corollary}
\begin{proof}
It follows from  Prop.~\ref{comphot} and Cor.~\ref{inchao}.
\end{proof}

\begin{remark}
    When \hott\ is extended with {\bf LEM} for h-propositions, it is possible to show that the type of h-propositions within any universe $\mathcal{U}$, that is $\mathsf{Prop}_{\mathcal{U}}$, is equivalent (and hence equal by Univalence) to the type of Booleans {\bf 2} in $\mathcal{U}$. Therefore, {\bf 2} behaves like a {\it subobject classifier} and we can prove that the category of h-sets, as defined in \cite{hottbook}, is an elementary topos.
\end{remark}

\begin{definition} A compatible translation from ${\bf HA}^\omega + {\bf iRC!_{\mathbf{N},\mathbf{N}}}$ to 
{\bf CZF}\ can be defined as follows: 

\begin{itemize}
    \item[-] $(\mathbf{N})^\ast\ \ourdef\ \omega$, while $(\sigma \rightarrow \tau)^\ast\ \ourdef\ \sigma^\ast \rightarrow \tau^\ast$, which is the set of functional relations from $\sigma^\ast$ to $\tau^\ast$, and $(\sigma \times \tau)^\ast\ \ourdef\ \sigma^\ast \times \tau^\ast$, which is the usual cartesian product of sets. 

    \item[-] terms of $\mathbf{HA}^\omega$ are interpreted by the corresponding definable terms of {\bf CZF}.

\item [-] the formulas of many-sorted intutionistic logic are interpreted by the corresponding formulas of intuitionistic logic, with sorts interpreted as predicates: e.g., $(\forall x^\sigma\ \phi(x))^\ast \ourdef \forall x\ ( x \in \sigma^\ast\ \rightarrow\ \phi^\ast(x))$.

\end{itemize}
\end{definition}
\begin{proposition}\label{compczf}
    The translation $(-)^\ast$ is a compatible translation of ${\bf HA}^\omega + {\bf iRC!_{\mathbf{N},\mathbf{N}}}$ in {\bf CZF}.
\end{proposition}

\begin{corollary}
    {\bf CZF} with the addition of the law of excluded middle becomes impredicative.
\end{corollary}
\begin{proof}
It follows from  Prop.~\ref{compczf} and Cor.~\ref{inchao}.
\end{proof}

It is worth noting that in the case of {\bf CZF}, the incompatibility with classical predicativity can be shown by a very direct and simple argument proved by Aczel in \cite{aczel78}:
\begin{proposition}
    The axiom of exponentiation and the law of excluded middle for $\Delta_0$-formulas imply the powerset axiom.
\end{proposition}
\begin{proof}
    It is enough to show that $\mathcal{P}(1)$ is a set. This follows from an application of {\bf LEM}.
    The rest follows from the fact that the powerset axiom is equivalent to exponentiation plus the existence of the powerset of the singleton set.
\end{proof}

Moreover, since the powerset axiom implies subset collection, which in turn implies exponentiation, it follows: 
\begin{proposition}
    {\bf CZF} with classical logic and {\bf ZF} prove the same set of theorems.
\end{proposition}
\begin{proof}
    See \cite{aczel78}.
\end{proof}

\section{A possible way out: the Minimalist Foundation}\label{sec6}

\noindent In this section, we want to show that there exists a constructive predicative system that does not present the same incompatibility with classical predicativity as the other systems previously considered and, at the same time, is enough expressive to allow for the development of non-trivial mathematics by keeping a form of {\it exponentiation} for a primitive notion of functions defined by $\lambda$-terms. This system is the {\it Minimalist Foundation} \mf.

\mf\  is a two-level foundation for constructive mathematics, that was first conceived in \cite{ms05} and then fully developed in \cite{m09}.
It consists of an intensional level, called \mtt,  and an extensional one, called \emtt, together with an interpretation of the latter into the first, as defined in \cite{m09}.

The two-level structure facilitates the compatibility of {\bf MF} with the most relevant constructive and classical foundations as shown in \cite{m09}: the intensional level \mtt\ is compatible with
\mltt\  and Coquand-Paulin's Calculus of Inductive Constructions \cite{Coqpau},
while the extensional level is compatible with {\bf CZF}, \hott\ and the internal theory  of elementary topoi in \cite{LS1986, Maietti05}.  It is a remarkable property that {\it both} levels of {\bf MF} are compatible with \hott\ as shown in
\cite{cm2024}.

Moreover, both levels of \mf\ extend a version of Martin-L{\"o}f's type theory with a primitive notion
of proposition: \mtt\ extends the intensional type theory in \cite{nordstrom1990programming}, while \emtt\ extends the extensional version presented in \cite{martin-lof:bibliopolis}.

\noindent In \mf\ there are four primitive distinct sorts: small propositions, propositions, sets and collections. The following diagram shows the relations between these primitive sorts

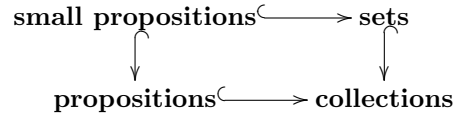
\begin{center}
$$
\xymatrix@C=0.5em@R=0.5em{
 {\mbox{\bf small propositions}} \ar@{^{(}->}[dd]\ar@{^{(}->}[rr]&  &
 {\mbox{\bf sets}} \ar@{^{(}->}[dd] \\ 
&&\\
  {\mbox{\bf propositions}} \ar@{^{(}->}[rr] && {\mbox{\bf collections}}
}
$$
\end{center}
\vspace{1.0em}

\noindent The basic forms of judgement in \mf\ include:
\begin{center}
    
 $ A \ set\ [\Gamma] \qquad B\ coll\ [\Gamma] 
\qquad \phi\ prop\
 [\Gamma]\qquad \psi\ prop_s\
 [\Gamma] 
 $
\end{center}
 to which we add the meta-judgement
$$A\ \type\ [\Gamma]$$
where `\type ' is to be interpreted as a meta-variable ranging over the four basic sorts.


The set-constructors of \mtt\ and \emtt\  include those of first order  Martin-L{\"o}f's type theory, respectively  as presented in \cite{nordstrom1990programming} and \cite{martin-lof:bibliopolis}. Hence we have the following set constructors: the empty set $N_0$, the singleton set
$N_1$, the set of lists  $List(A)$ over a set $A$, hence the set of natural numbers $\mathsf{N}$, the indexed sum and the dependent product of the family of sets $B(x)\ set \ [x\in A]$ denoted respectively as $\Sigma_{x\in A}\ B(x)$  and  $\Pi_{x\in A}\ B(x)$, the disjoint sum $A+B$ of the set $A$ with the set $B$. In the case of the extensional level \emtt\, we have that sets are closed under effective quotients $A/R$ over a set $A$, provided that $R $ is a small equivalence relation $R(x,y) \
prop_s\ [x\in A,y\in A]$. Such quotient constructors are not included in the intensional level \mtt. Furthermore, both \mtt\ and  \emtt\ include small propositions $\phi\  prop_s$ regarded as sets of their proofs and defined are those propositions closed under intuitionistic connectives, quantifiers, and equalities restricted to sets. 

\noindent Collections of \mtt\ and \emtt\ include their sets, the indexed sum  $\Sigma_{x\in A} B(x)$  of the family of collections $B(x)\  coll \ [x\in A]$ indexed over a collection $A$, and propositions $\psi\ prop$ regarded as collections of their proofs. A significant distinction between \mtt-collections and \emtt-collections is the following: the former includes the proper collection of small propositions $\mathsf{prop_s}$ and  
the collection of small propositional functions  $A\rightarrow  \mathsf{prop_s}$ over a set $A$ (which  are never sets  predicatively when $A$ is not empty), while the latter includes
the power-collection of the singleton  $ {\cal P}(1)$, which is the quotient of the collection of small propositions under the relation of equiprovability,  and  
the power-collection $A\rightarrow  {\cal P}(1)$ of a set $A$, that can be written simply as ${\cal P}(A)$.

\noindent In \mtt\  we identify {\it predicates over the set $A$}
with {\it propositional functions} defined as primitive {\it $\lambda$-terms}. Hence, in \emtt\
a subset of a set $A$ is identified with the equivalence class of propositional function up to equiprovability. 
So this use of $\lambda$-terms corresponding to predicates is a major difference with an axiomatic set theory where subsets of a set $A$ are simply defined as sets
in a certain relation with $A$ and hence are identified with functional relations into the boolean set $\{0,1\}$ when the logic is classical.

\noindent Both propositions of \mtt\ and \emtt\ include standard connectives and quantifiers. However, they differ in the way the rules for propositional equality are given. Indeed, \mtt-propositions include an {\it intensional} propositional equality defined as in \cite{nordstrom1990programming}, except for the fact that the elimination rule is {\it restricted} to propositions. While \emtt-propositions include an {\it extensional} propositional equality with an elimination rule formulated like that in \cite{martin-lof:bibliopolis}.

\noindent Finally, we recall that \emtt-propositions are {\it proof-irrelevant} and have one canonical element denoted as $\mathsf{true}$. 


\noindent It is relevant for our discussion to remark that elimination rules for propositions act only towards propositions and not towards collections. The same applies to small propositions. In this way, \mtt\ and \emtt\ do not generally validate choice principles. In particular, even unique choice cannot be derived in both levels of  \mf, thanks to results in \cite{qu12,MR16,Ma17} and the non-derivability of unique choice in Coquand's Calculus of Constructions {\bf CoC}, as first proved in \cite{st92}. This is in contrast with what happens in \mltt\ and in \hott, as we have seen in Section \ref{sec3}. The fact that choice principles are not validated in \mf\ is important in connection with the notion of compatibility introduced in Section \ref{sec5}. In this way, \mf\ turns out to be also compatible with theories where choice does not generally hold, such as {\bf CoC}.

Now, we discuss the status of the internal rule of unique choice in \mf, in particular in \mtt.

\begin{definition}
    In \mtt\ the {\it internal rule of unique choice}  holds if for any derivable small proposition
    \begin{equation*}
        R(x,y)\ prop_s\ [ x \in A, y\in B]
    \end{equation*}
    and for any derivable judgement of the form 
   \begin{equation*}
        p(x) \in \exists !_{ y\in B}\  R(x,y)\ [x \in A]
    \end{equation*}
 there exists a proof-term $q$ such that the following judgement is derivable
\begin{equation*}
    q \in\exists_{f\in A\rightarrow B}\ \forall_{x\in A} \ R(\, x,\mathsf{Ap}(f,x)\, )\ \
\end{equation*}

We call  {\it number-theoretic internal rule of unique choice} the instance of the internal rule of unique choice when $A$ and $B$ coincide with the set $\mathsf{N}$ of natural numbers.

Instead, we reserve the name {\it rule of unique choice} for the version of the above rule requiring
the external existence of the choice function term $f$, i.e. if $\exists !_{ y\in B}\  R(x,y)\ [x \in A]$ holds then we can derive a term $f\in A\rightarrow B$ such that
$R(\, x,\mathsf{Ap}(f,x)\, )\ [x\in A]
$ holds.
\end{definition}

Concerning this rule observe that:
\begin{proposition}\label{norun}
The rule of unique choice
is not valid in \mtt. 
\end{proposition}
\begin{proof}
See proof of Th.1 of \cite{Ma17}.
\end{proof}
The existence property of \mtt\ would allow
to deduce the equivalence of the rule of unique choice with its internal version and
hence the non-validity of the internal rule of unique choice in \mtt, too.

\noindent \noindent  Nevertheless, it is worth pointing out that without
relying on the existence property, we can show the non-validity of the internal rule of unique choice following
the same argument in \cite{Ma17} since this can also be exported to
extensions of \mtt,  not validating the existence property,  such as its classical version:
\begin{proposition}\label{rulax}
    If \mtt\ satisfies the internal rule of unique choice, then it satisfies the axiom of unique choice. 
\end{proposition}
\begin{proof}
The proof of Prop.3  in \cite{Ma17} can be easily adapted to the case of the internal rule of unique choice.
\end{proof}
Now, thanks to the model in \cite{st92}, we deduce from Prop.~\ref{rulax}  that
\begin{corollary}\label{coruc}
    \mtt\ does not validate the internal rule of unique choice. 
\end{corollary}

\noindent It is important to remark that already the intensional level \mtt\ can be regarded as an extension of ${\bf HA}^\omega$. Indeed, it is possible to show that there is a compatible translation of this theory into \mtt\ following a similar translation of Section \ref{sec5} for \mltt\ by interpreting ${\bf HA}^\omega$-propositions as \mtt-propositions. What is most interesting is that \mf\  extends ${\bf HA}^\omega$ while preserving the fundamental property of equi-consistency with its classical counterpart. Indeed, recently it has been shown that {\it both levels of \mf\ are equiconsistent with their classical counterparts} (collapsed to the extension of  the extensional level \emtt\ with the law of excluded middle), unlike theories such as \mltt, \hott\ and $\textbf{CZF}$: 

\begin{proposition}\label{equi} Both levels of \mf\ are equiconsistent with \emtt\ extended  the addition of the law of excluded middle $\phi \vee \neg \phi$ for all propositions $\phi$.
\end{proposition} 
\begin{proof}
For a proof, we refer to \cite{MSdoubleneg}.
\end{proof}

Furthermore, from \cite{m09,IMMS} we know a proof-theoretic upper bound for \mf:
\begin{proposition}\label{boun} The proof-theoretic strength of both levels of \mf\ are bounded by Feferman's theory $\widehat{ID_1}$ of non-iterative fixpoints.
\end{proposition} 
Therefore, the extension of \mf\ with the law of excluded middle is still predicative:

\begin{corollary}\label{bouncla} The proof-theoretic strength of the classical extension of \mf\ is  bounded by Feferman's theory $\widehat{ID_1}$ of non-iterative fixpoints.
\end{corollary}
\begin{proof} It follows from Prop.\ref{equi} and  Prop.\ref{boun}.

From this, knowing that  {\bf HA$^\omega$} can be  embedded in \mtt\ and \emtt\ 
in a compatible way, we can deduce the non-derivability of  unique choice on natural numbers in their
classical version:
\begin{corollary}\label{norucla}
The classical versions of \emtt\ and \mtt\ do not validate the number-theoretic
internal rule of choice and hence the number-theoretic axiom of unique choice:
\begin{equation*}\label{uniquenat}
 ({\bf AC!_{\mathsf{N},\mathsf{N}}})\    \forall_{x\in \mathsf{N}}\ \exists !_{y\in \mathsf{N}}\ R(x,y)\ \rightarrow\ \exists_{f\in \mathsf{N} \rightarrow \mathsf{N}}\ \forall_{x\in \mathsf{N}}\ R(x,f(x))
\end{equation*}
\end{corollary}
\begin{proof}
First, observe that $\bf AC!_{\mathsf{N},\mathsf{N}}$ implies the number-theoretic internal rule of unique choice. The claim then follows from Proposition~\ref{bouncla} and Cor.~\ref{inchao}.
\end{proof}
Since the validity of the axiom of unique choice  in a theory can be transferred to its extensions, contrary to the rule of unique choice, 
from Cor.~\ref{norucla} we can immediately deduce the non-validity of ${\bf AC!_{\mathsf{N},\mathsf{N}}}$
in \mtt\ and \emtt\ themselves:
\begin{corollary}
Both \emtt\ and of \mtt\ do not validate the number-theoretic axiom of unique choice.
\end{corollary}

\noindent It is worth noting that in the extensional level \emtt\ of \mf\ we can also interpret the intuitionistic version of {\bf ACA}  in a way that preserves the meaning of all logical operators:

\begin{proposition}\label{acint}
The  intuitionistic version of {\bf ACA} is interpretable in the extensional level \emtt\  of \mf. Hence,  (the classical) {\bf ACA} can be interpreted in the classical version of \emtt.
\end{proposition}
\begin{proof}
It is sufficient to define a translation from {\bf ACA} to \emtt\ as follows. Numerical variables of {\bf ACA} are translated as variables of type $\mathsf{N}$ in \emtt. Set variables of {\bf ACA} are translated as terms of type $ \mathcal{P}(\mathsf{N})$. The symbols for zero and the successor in {\bf ACA} are interpreted by the corresponding constructors for $\mathsf{N}$ in \emtt. Atomic formulas of the form $s = t$ are translated as propositional equalities over $\mathsf{N}$ and those of the form $t \in X$ as small propositions $t\ \varepsilon\ X$\footnote{This notation can be defined in \emtt. Given $X \in \mathcal{P}(A)$ and $t \in A$, we define $t\ \varepsilon\ X$ as $\mathsf{Eq}(\mathcal{P}(1), X(t), [ \top])$. Moreover, it holds that $t\ \varepsilon\ \{x \in A \mid \phi(x) \} \leftrightarrow \phi(t) $ is true. }. In general, first-order formulas of {\bf ACA} are translated as small propositions with quantifiers ranging over $\mathsf{N}$, while second-order formulas are translated as propositions with quantifiers ranging over the collection $\mathcal{P}(\mathsf{N})$. 

Peano axioms can be derived in \emtt\ thanks to the rules for $\mathsf{N}$ and the presence of $\mathcal{P}(1)$. In particular, the induction scheme of {\bf ACA} follows from the elimination rule for $\mathsf{N}$, which acts towards all collections, including propositions. 

Finally, the arithmetical comprehension axiom is validated as follows: given a small proposition $\phi(x)\ [x \in \mathsf{N}]$, we can introduce $[\phi(x)]\in \mathcal{P}(1)\ [x \in \mathsf{N}]$ and then form $\{ x \in \mathsf{N}\ \mid\ \phi(x)\} \in \mathcal{P}(\mathsf{N})$. Then, it follows that there is an $X \in \mathcal{P}(\mathsf{N})$ such that for all $x \in \mathsf{N}$, $x\ \varepsilon\ X\ \leftrightarrow\ \phi(x)$ holds.
\end{proof}
Hence, from  prop.~\ref{acint} it follows that the classical version of \mf\ can formalize Weyl's classical predicative reformulation of analysis in {\it Das Kontinuum} following Feferman's analysis in \cite{feferman1998light, feferman2000significance}.

\end{proof}

\noindent Another significant feature of \mf\, which makes it compatible with Weyl's classical predicativity, is that {\it real numbers do not form a set} in its extensional level \emtt\ (and this holds also for their intensional representation in the quotient model over \mtt). This reflects Weyl's idea that the collection of real numbers is not an extensionally determinate collection.

\begin{proposition}
   In  the extensional level \emtt\  of \mf\ real numbers à la Dedekind or Cauchy do not form a set as well as in their classical version.
\end{proposition}
\begin{proof}
A direct proof of this was given in \cite{ober2020}
but it follows also from the predicativity of the classical version of \mf\ (see also \cite{MSdoubleneg}). If real numbers form a set in \mf\ then this is so in the classical version of \mf\
and every interval of real numbers form a set too. Since classically,  the power-collection of natural numbers can be put in bijection with an interval of real numbers, then it would follow that impredicative definitions can be employed in the classical version of \mf\ contrary to its predicativity established 
 thanks in prop.\ref{bouncla}.  
\end{proof}
Despite the absence of choice principles and quantification over all real numbers to form new sets,
a  development of  mathematics, and in particular of topology and analysis, within \emtt\
is possible if    point-free methods are adopted as advocated by Martin-L{\"o}f and Sambin
with the introduction of Formal Topology in \cite{sambin1987}.
Indeed, in \emtt\ the following point-free topologies are definable: 
\begin{proposition}
 In  the extensional level \emtt\  of \mf\ the point-free topology of Dedekind real numbers and Cantor space are definable.
\end{proposition}
\begin{proof}
The formalization of the point-free topology of the Cantor space can be carried on in \emtt\ as in section 2.1 of \cite{Valentini-bar}.
Then, observe that the point-free topology of real numbers can be defined as in Def.3 of \cite{CN95}
in terms of another cover generated by a finite set of axioms that can be represented in \emtt\ with a completely analogous proof to that just mentioned for the point-free topology of  Cantor space.
\end{proof}
\noindent Therefore, one can carry on in \mf\ the development of analysis as pursued in \cite{CN95,Palmgren05,pal07} (see also \cite{kawai2023}) and that of Positive Topology in \cite{sambinbook} and {\it loc.cit.}.

\section{Conclusion}
We have discussed the incompatibility  with Weyl's classical predicativity
of some foundations for Bishop mathematics like Aczel's constructive set theory {\bf CZF}, Martin-L\"of Type Theory \mltt\ and  Homotopy Type Theory \hott.
 We have observed that this incompatibility reduces to the fact that these theories interpret Heyting arithmetic in all finite types ${\bf HA^\omega}$, which is a many-sorted theory where sorts include finite types,  with the addition of the internal rule of number-theoretic unique choice {\bf iRC!$_{\mathbf{N},\mathbf{N}} $}, identifying functional relations over natural numbers with a primitive notion of function as  $\lambda$-terms. Indeed,
 {\bf HA$^{\omega}$} + {\bf iRC!$_{\mathbf{N}, \mathbf{N}}$}
 with the further addition of the law of excluded middle becomes impredicative since it interprets second-order Peano arithmetic.

\noindent Whilst Weyl's foundational system does not contain quantification over exponentiation of functions as observed by \cite{DBLP:journals/bsl/Avron20}, we have argued
 that it is not necessary to sacrifice exponentiation altogether  to reconcile Bishop's mathematics with Weyl's approach if we adopt a foundation like
the Minimalist Foundation.

Given that the Minimalist Foundation interprets {\bf HA$^{\omega}$}, 
the price to pay is to renounce the rule of unique choice as well as
all the number-theoretic choice principles characteristic of Bishop's mathematics and that real numbers à la Dedekind or Cauchy form a set.
All these facts together call for a point-free  development of  analysis by adopting 
  the topological methods advocated by Martin-L\"of and Sambin with the introduction of Formal Topology \cite{sambin1987}.

Hence, provided that a point-free reformulation of classical analysis is viable, as hinted in \cite{CN95,Palmgren05, pal07,kawai2023},
 \mf\  promises to be a natural crossroads between Bishop's constructivism and Weyl's classical predicativity.

In the future, besides establishing the exact proof-theoretic strength of  \mf,  which is currently an open problem,
we will explore whether the extensions of \mf\   with inductive and coinductive definitions in \cite{MMR21, MMR22, MS2023} 
are still  equiconsistent with their classical counterpart, as established for \mf\ in \cite{MSdoubleneg}, or at least whether they remain predicative after the addition of the law of excluded middle.
A positive answer to either of these questions would provide a more expressive constructive foundation capable of formalizing  general results of point-free mathematics as those presented in
\cite{CSSV03,mv04,curicind,sambinbook},
by keeping the 
compatibility both with Bishop's constructivism and with Weyl's classical predicativity, as in the case of \mf. 

\bibliographystyle{abbrv}
\bibliography{biblio}

@inproceedings{Ma17,
  author    = {Maria Emilia Maietti},
  title     = {On Choice Rules in Dependent Type Theory},
  booktitle = {{TAMC}},
  series    = {Lecture Notes in Computer Science},
  volume    = {10185},
  pages     = {12--23},
  year      = {2017}
}

@misc{palmgren2005bishop,
  title={Bishop’s set theory},
  author={Palmgren, Erik},
  note={Slides for lecture at the TYPES summer school},
  year={2005}
}

@incollection{aczel78,
title={The type theoretic interpretation of constructive set theory},
author={P. Aczel},
booktitle={ Logic Colloquium `77}, 
editor={MacIntyre, A. and Pacholski, L. and Paris, J.},
series={Studies in Logic and the Foundations of Mathematics},
volume ={96},
pages={55--66},
year={1978},
publisher={North Holland}
}

@article{spa,
  author    = {F. Ciraulo and     G. Sambin},
  title     = {Reducibility, a constructive dual of spatiality},
  journal   = {Journal of Logic and Analysis},
  volume    = {11},
  year      = {2019}
}

@article{mmr21,
      title={A realizability semantics for inductive formal topologies, {C}hurch's {T}hesis and {A}xiom of {C}hoice}, 
      author={Maria Emilia Maietti and S. Maschio and M. Rathjen},
        journal= 	 {Logical Methods in Computer Science},
        year={2021},
      volume={17},
      number={2}}

@inproceedings{CN95,
  author       = {Jan Cederquist and
                  Sara Negri},
  editor       = {Stefano Berardi and
                  Mario Coppo},
  title        = {A {C}onstructive {P}roof of the {H}eine-{B}orel {C}overing {T}heorem for {F}ormal
                  Reals},
  booktitle    = {Types for Proofs and Programs, International Workshop TYPES'95, Torino,
                  Italy, June 5-8, 1995, Selected Papers},
  series       = {Lecture Notes in Computer Science},
  volume       = {1158},
  pages        = {62--75},
  publisher    = {Springer},
  year         = {1995},
  url          = {https://doi.org/10.1007/3-540-61780-9\_62},
  doi          = {10.1007/3-540-61780-9\_62},
  bibsource    = {dblp computer science bibliography, https://dblp.org}
}

@article {IMMS,
    AUTHOR = {H. Ishihara and Maria Emilia Maietti and S. Maschio and T. Streicher},
     TITLE = {Consistency of the intensional level of the {M}inimalist
              {F}oundation with {C}hurch's {T}hesis and {A}xiom of {C}hoice},
   JOURNAL =  {Archive for Mathematical Logic},
    VOLUME = {57},
      YEAR = {2018},
    NUMBER = {7-8},
     PAGES = {873--888}}

@inproceedings{MR16,
	Author = {Maria Emilia {Maietti} and G. {Rosolini}},
	Booktitle = {Concepts of Proof in Mathematics, Philosophy, and Computer Science},
	Editor = {Dieter Probst and Peter Schuster},
	Title = {Relating quotient completions via categorical logic.},
	pages={229-250},
	Year ={2016}}

@article{qu12,
	Author = {Maria Emilia {Maietti} and G. {Rosolini}},
	Journal = {Logica Universalis},
	Number = {3},
	Pages = {371-402},
	Title = {Quotient completion for the foundation of constructive mathematics},
	Volume = {7},
	Year = {2013}}

@article{m09,
	Author = {Maria Emilia {Maietti}},
	Journal = {Annals of Pure and Applied Logic},
	Number = {3},
	Pages = {319-354},
	Title = {A minimalist two-level foundation for constructive mathematics.},
	Volume = {160},
	Year = {2009}}

@inproceedings{ms05,
	Author = {Maria Emilia {Maietti} and Giovanni {Sambin}},
	Booktitle = {From Sets and Types to Topology and Analysis: Practicable Foundations for Constructive Mathematics},
	Editor = {{L. Crosilla and P. Schuster}},
	Number = {48},
	Pages = {91-114},
	Publisher = {{Oxford University Press}},
	Series = {{Oxford Logic Guides}},
	Title = {{Toward a minimalist foundation for constructive mathematics}},
	Year = {2005}}

@unpublished{czf,
	Author = {P. {Aczel} and M. {Rathjen}},
	Note = {Mittag-Leffler Technical Report No.40},
	Title = {Notes on constructive set theory.},
	Year = {2001}}

@article{rijke2015sets,
  title={Sets in homotopy type theory},
  author={Rijke, E. and Spitters, B.},
  journal={Mathematical Structures in Computer Science},
  volume={25},
  number={5},
  pages={1172--1202},
  year={2015},
  publisher={Cambridge University Press}
}

@Book{hottbook,
  author =    {{Univalent Foundations Program}},
  title =     {Homotopy Type Theory: Univalent Foundations of Mathematics},
  publisher = {https://homotopytypetheory.org/book},
  address =   {Institute for Advanced Study},
  year =      2013}

@article{CSSV03,
  author    = {T. Coquand and
               G. Sambin and
               J.M. Smith and
               S. Valentini},
  title     = {Inductively generated formal topologies},
  journal   = {Ann. Pure Appl. Log.},
  volume    = {124},
  number    = {1-3},
  pages     = {71--106},
  year      = {2003},
  url       = {https://doi.org/10.1016/S0168-0072(03)00052-6},
  doi       = {10.1016/S0168-0072(03)00052-6},
  bibsource = {dblp computer science bibliography, https://dblp.org}
}

@article{st92,
  title={Independence of the induction principle ad the axiom of choice in the pure calculus of constructions},
  author={Streicher, T.},
  journal={Theoretical computer science},
  volume={103},
  number={2},
  pages={395--408},
  year={1992},
  publisher={Elsevier}
}

@book{Weyl+1918,
title = {Das Kontinuum. Kritische Untersuchungen über die Grundlagen der Analysis},
author = {Hermann Weyl},
publisher = {De Gruyter},
address = {Berlin, Boston},
doi = {doi:10.1515/9783112451144},
isbn = {9783112451144},
year = {1918}
}

@incollection{ober2020,
    author = {Maria Emilia Maietti},
    title = "On minimality of the {M}inimalist {F}oundation",
    booktitle ="Mathematisches Forschungsinstitut Oberwolfach
Report No. 34/2020, Mathematical Logic: Proof Theory, Constructive Mathematics (hybrid meeting)" ,
    publisher ={MFO} ,
    year = {2020},
url={https://publications.mfo.de/bitstream/handle/mfo/3821/OWR_2020_34.pdf?sequence=4&isAllowed=y}
}

@inproceedings{sambin2008continuity,
  title={Two applications of dynamic constructivism: Brouwer’s continuity principle and choice sequences in formal topology},
  author={Sambin, Giovanni},
editor={Mark Van Atten, Pascal Boldini, Michel Bourdeau and Gerhard Heinzmann},
  booktitle={One Hundred Years of Intuitionism (1907--2007) The Cerisy Conference},
  pages={301--315},
  year={2008},
  organization={Springer}
}

@book{bishop1967,
  title={Foundations of constructive analysis},
  author={Bishop, Errett},
  volume={60},
  year={1967},
  publisher={McGraw-Hill}
}

@book{sambinbook,
    author = {Giovanni Sambin},
    title = {Positive Topology: A New Practice in Constructive Mathematics},
    publisher = {Oxford University Press},
    year = {2025},
}

@incollection{Crosilla2022,
	author = {Laura Crosilla},
	booktitle = {Objects, Structures and Logics},
	editor = {Gianluigi Oliveri and Claudio Ternullo and Stefano Boscolo},
	publisher = {Springer Cham},
	title = {Predicativity and {C}onstructive {M}athematics},
	year = {2022}
}

@article {Valentini-bar,
    AUTHOR = {S. Valentini},
     TITLE = {Constructive characterizations of bar subsets},
   JOURNAL = {Ann. Pure Appl. Logic},
  FJOURNAL = {Annals of Pure and Applied Logic},
    VOLUME = {145},
      YEAR = {2007},
    NUMBER = {3},
     PAGES = {368--378}
}

@article{curicind,
  author    = {G. Curi},
  title     = {Abstract {I}nductive and {C}o-{I}nductive {D}efinitions},
  journal   = {J. Symb. Log.},
  volume    = {83},
  number    = {2},
  pages     = {598--616},
  year      = {2018}
}

@article{mv04,
	Author = {Maria Emilia {Maietti} and S. {Valentini}},
	Journal = {Journal of Symbolic Logic},
	Pages = {967-1005},
	Title = {A structural investigation on formal topology: coreflection of formal covers and exponentiability},
	Volume = {69},
	Year = {2004}}

@incollection{talkmillyober,
author       = {Maria Emilia Maietti},
title={Equiconsistency of the {M}inimalist {F}oundation with its classical version.},
booktitle={Mathematisches Forschungsinstitut Oberwolfach Report n.50/2023- Mathematical Logic: Proof Theory,
Constructive Mathematics},
publisher={MFO},
year={2024}
}

@inproceedings{Coqpau,
  author       = {Thierry Coquand and
                  Christine Paulin},
  editor       = {Per Martin{-}L{\"{o}}f and
                  Grigori Mints},
  title        = {Inductively defined types},
  booktitle    = {COLOG-88, International Conference on Computer Logic, Tallinn, USSR,
                  December 1988, Proceedings},
  series       = {Lecture Notes in Computer Science},
  volume       = {417},
  pages        = {50--66},
  publisher    = {Springer},
  year         = {1988},
  url          = {https://doi.org/10.1007/3-540-52335-9\_47},
  doi          = {10.1007/3-540-52335-9\_47},
  bibsource    = {dblp computer science bibliography, https://dblp.org}
}

@article{cm2024,
  title={The {C}ompatibility of the {M}inimalist {F}oundation with {H}omotopy {T}ype {T}heory},
  author={Contente, Michele and Maietti, Maria Emilia},
  journal={Theoretical Computer Science},
  volume={991},
  pages={114421},
  year={2024},
  publisher={Elsevier}
}

@incollection{Palmgren05,
  author       = {Erik Palmgren},
  editor       = {Laura Crosilla and
                  Peter M. Schuster},
  title        = {Continuity on the real line and in formal spaces},
  booktitle    = {From {S}ets and {T}ypes to {T}opology and {A}nalysis},
  series       = {Oxford Logic Guides},
  volume       = {48},
  publisher    = {Oxford University Press},
  year         = {2005},
  bibsource    = {dblp computer science bibliography, https://dblp.org}
}

@article{CPS11,
  author       = {Thierry Coquand and
                  Erik Palmgren and
                  Bas Spitters},
  title        = {Metric complements of overt closed sets},
  journal      = {Math. Log. Q.},
  volume       = {57},
  number       = {4},
  pages        = {373--378},
  year         = {2011},
  url          = {https://doi.org/10.1002/malq.201010011},
  doi          = {10.1002/MALQ.201010011},
  bibsource    = {dblp computer science bibliography, https://dblp.org}
}

@article{pal07,
title = {A constructive and functorial embedding of locally compact metric spaces into locales},
journal = {Topology and its Applications},
volume = {154},
number = {9},
pages = {1854-1880},
year = {2007},
issn = {0166-8641},
doi = {https://doi.org/10.1016/j.topol.2007.01.018},
url = {https://www.sciencedirect.com/science/article/pii/S0166864107001113},
author = {Erik Palmgren}
}

@incollection{kawai2023,
  title={Bishop {M}etric {S}paces in {F}ormal {T}opology},
  author={Kawai, Tatsuji},
  booktitle={Handbook of Constructive Mathematics},
  editor = {Douglas Bridges, Hajime Ishihara, Michael Rathjen and Helmut Schwichtenberg},
  pages={395--425},
  year={2023},
  publisher={Cambridge University Press}
}

@incollection{maiettisambinhand,
  title={The {M}inimalist {F}oundation and {B}ishop’s {C}onstructive {M}athematics},
  author={Maietti, Maria Emilia and Sambin, Giovanni},
  booktitle={Handbook of Constructive Mathematics},
 editor = {Douglas Bridges, Hajime Ishihara, Michael Rathjen and Helmut Schwichtenberg},
  pages={525--563},
  year={2023},
  publisher={Cambridge University Press}
}

@book{troelstra1973,
  title={Metamathematical investigation of intuitionistic arithmetic and analysis},
  author={Troelstra, Anne Sjerp},
  year={1973},
  publisher={Springer}
}

@article{MSpointfree,
	author = {Maria Emilia Maietti and Giovanni Sambin},
	doi = {10.12775/llp.2013.010},
	journal = {Logic and Logical Philosophy},
	number = {2},
	pages = {167--199},
	publisher = {Nicolaus Copernicus University Scientific Publishing House},
	title = {Why {T}opology in the {M}inimalist {F}oundation {M}ust {B}e {P}ointfree},
	volume = {22},
	year = {2013},
}

@incollection{spector1962provably,
  title={Provably recursive functionals of analysis: a consistency proof of analysis by an extension of principles formulated in current intuitionistic mathematics},
  author={Spector, Clifford},
  booktitle={Recursive function theory: Proceedings of Symposia in Pure Mathematics},
volume = {5},
editor={Dekker, F.D.E.},
  pages={1--27},
  year={1962},
  publisher={American Mathematical Society}
}

@article{goodman1978,
  title={Relativized realizability in intuitionistic arithmetic of all finite types},
  author={Goodman, Nicolas D},
  journal={The Journal of Symbolic Logic},
  volume={43},
  number={1},
  pages={23--44},
  year={1978},
  publisher={Cambridge University Press}
}

@inproceedings{MS2023,
  author       = {Maria Emilia Maietti and
                  Pietro Sabelli},
  editor       = {Marie Kerjean and
                  Paul Blain Levy},
  title        = {A topological counterpart of well-founded trees in dependent type
                  theory},
  booktitle    = {Proceedings of the 39th Conference on the Mathematical Foundations
                  of Programming Semantics, {MFPS} XXXIX, Indiana University, Bloomington,
                  IN, USA, June 21-23, 2023},
  series       = {{EPTICS}},
  volume       = {3},
  publisher    = {EpiSciences},
  year         = {2023},
  url          = {https://doi.org/10.46298/entics.11755},
  doi          = {10.46298/ENTICS.11755},
  bibsource    = {dblp computer science bibliography, https://dblp.org}
}

@article{MMR22,
  author       = {Maria Emilia Maietti and
                  Samuele Maschio and
                  Michael Rathjen},
  title        = {Inductive and {C}oinductive {T}opological {G}eneration with {C}hurch's thesis
                  and the {A}xiom of {C}hoice},
  journal      = {Log. Methods Comput. Sci.},
  volume       = {18},
  number       = {4},
  year         = {2022},
  url          = {https://doi.org/10.46298/lmcs-18(4:5)2022},
  doi          = {10.46298/LMCS-18(4:5)2022},
  bibsource    = {dblp computer science bibliography, https://dblp.org}
}

@incollection{sambin1987,
  title={Intuitionistic formal spaces - a first communication},
  author={Sambin, Giovanni},
editor = {Dimiter G. Skordev}, 
  booktitle={Mathematical logic and its Applications},
  pages={187--204},
  year={1987},
  publisher={Springer}
}

@book{TVD88,
series = {Studies in logic and the foundations of mathematics 121/123},
publisher = {North Holland},
booktitle = {Constructivism in {M}athematics: {A}n {I}ntroduction},
year = {1988},
title = {Constructivism in Mathematics : An Introduction},
language = {eng},
address = {Amsterdam},
author = {Troelstra, Anne Sjerp and van Dalen, Dirk}
}

@book{weyl1994continuum,
  title={The Continuum: a critical examination of the foundation of analysis},
  author={Weyl, Hermann},
  year={1994},
  publisher={Courier Corporation}
}

@inproceedings{feferman1988weyl,
 title={\lq {W}eyl {V}indicated: {D}as {K}ontinuum, 70 {Y}ears {L}ater'},
  author={Feferman, Solomon},
  year={1988},
editor={Carlo Cellucci and Giovanni Sambin},
  booktitle={ Atti del Cogresso Temi e Prospettive della Logica e della Filosofia delle Scienze Contemporanee, SILFS}, 
  publisher={Bologna: CLUEB, 59--93},
 }

@book{feferman1998light,
  title={In the Light of Logic},
  author={Feferman, Solomon},
  year={1998},
  publisher={Oxford University Press}
}

@incollection{feferman2000significance,
 title={\lq {T}he significance of {W}eyl’s {D}as {K}ontinuum'},
  author={Feferman, Solomon},
  booktitle={Proof Theory -- History and Philosophical Significance},
editor={Vincent F. Hendricks and Stig Andur Pedersen and Klaus Frovin Jorgensen },
  year={2000},
  publisher={Dordrecht--Boston--London: Kluwer, 179--194}
}

@book{simpson2009subsystems,
  title={Subsystems of Second Order Arithmetic},
  author={Simpson, Stephen G},
  volume={1},
edition={2nd},
collection={Perspectives in Logic},
  year={2009},
  publisher={Cambridge University Press}
}

@article{adams2010weyl,
  title={Weyl's predicative classical mathematics as a logic-enriched type theory},
  author={Adams, Robin and Luo, Zhaohui},
  journal={ACM Transactions on Computational Logic (TOCL)},
  volume={11},
  number={2},
  pages={1--29},
  year={2010},
  publisher={ACM New York, NY, USA}
}

@article{DBLP:journals/apal/AdamsL10,
  author    = {Robin Adams and
               Zhaohui Luo},
  title     = {Classical predicative logic-enriched type theories},
  journal   = {Ann. Pure Appl. Log.},
  volume    = {161},
  number    = {11},
  pages     = {1315--1345},
  year      = {2010},
  url       = {https://doi.org/10.1016/j.apal.2010.04.005},
  doi       = {10.1016/j.apal.2010.04.005},
  bibsource = {dblp computer science bibliography, https://dblp.org}
}

@unpublished{AczelRath,
    author = {Peter Aczel and Michael Rathjen},
    title = {{\bf CST} {B}ook {D}raft},
year = {2010},
    note = {unpublished manuscript},
}

@article{DBLP:journals/bsl/Avron20,
  author    = {Arnon Avron},
  title     = {Weyl {R}eexamined: \lq\lq{D}as {K}ontinuum" 100 {Y}ears {L}ater},
  journal   = {Bull. Symb. Log.},
  volume    = {26},
  number    = {1},
  pages     = {26--79},
  year      = {2020},
  url       = {https://doi.org/10.1017/bsl.2020.23},
  doi       = {10.1017/bsl.2020.23},
  bibsource = {dblp computer science bibliography, https://dblp.org}
}

@incollection{ml1975,
  title={An intuitionistic theory of types: Predicative part},
  author={Martin-L{\"o}f, Per},
  booktitle={Studies in Logic and the Foundations of Mathematics},
  volume={80},
  pages={73--118},
  year={1975},
  publisher={Elsevier}
}

@book{nordstrom1990programming,
  title={Programming in Martin-L{\"o}f’s type theory},
  author={Nordstr{\"o}m, Bengt and Petersson, Kent and Smith, Jan M},
  volume={200},
  year={1990},
  publisher={Oxford University Press Oxford}
}

@book{martin-lof:bibliopolis,
	Author = {Martin-L{\"o}f, Per},
	Isbn = {88-7088-105-9},
	Mrclass = {03B15 (03F50 03F55)},
	Mrnumber = {769301 (86j:03005)},
	Mrreviewer = {M. M. Richter},
	Pages = {iv+91},
	Publisher = {Bibliopolis},
	Series = {Studies in Proof Theory},
	Subtitle = {Notes by Giovanni Sambin},
	Title = {Intuitionistic type theory},
	Volume = {1},
	Year = {1984}}

@article{MSdoubleneg,
title = {Equiconsistency of the {M}inimalist {F}oundation with its classical version},
journal = {Annals of Pure and Applied Logic},
volume = {176},
number = {2},
pages = {103524},
year = {2025},
issn = {0168-0072},
doi = {https://doi.org/10.1016/j.apal.2024.103524},
url = {https://www.sciencedirect.com/science/article/pii/S0168007224001283},
author = {Maria Emilia Maietti and Pietro Sabelli},
}

@article{crosilla2023weyl,
  title={Weyl and two kinds of potential domains},
  author={Crosilla, Laura and Linnebo, {\O}ystein},
  journal={No{\^u}s},
volume = {58},
pages={409-430},
  year={2024},
  publisher={Wiley Online Library}
}

@article{Maietti05,
  author       = {Maria Emilia Maietti},
  title        = {Modular correspondence between dependent type theories and categories
                  including pretopoi and topoi},
  journal      = {Math. Struct. Comput. Sci.},
  volume       = {15},
  number       = {6},
  pages        = {1089--1149},
  year         = {2005},
  url          = {https://doi.org/10.1017/S0960129505004962},
  doi          = {10.1017/S0960129505004962},
  bibsource    = {dblp computer science bibliography, https://dblp.org}
}

@book{LS1986,
	author = {Joachim Lambek and Philip J. Scott},
	publisher = {Cambridge University Press},
	title = {Introduction to Higher Order Categorical Logic},
	year = {1986}
}

\end{document}